\documentclass[11pt]{article}

\usepackage[affil-it]{authblk}

\usepackage{amsmath,amsfonts,amsthm,amssymb,verbatim,graphicx,color}

\theoremstyle{plain}
\newtheorem{claim}[]{Claim}
   \newtheorem{theorem}[equation]{Theorem}
   \newtheorem{corollary}[equation]{Corollary}
   \newtheorem{lemma}[equation]{Lemma}
   
\theoremstyle{definition}
   
\theoremstyle{remark}
   \newtheorem{remark}[equation]{Remark}

\DeclareMathOperator{\conv}{conv}
\DeclareMathOperator{\cone}{cone}

\DeclareMathOperator{\qstab}{QSTAB}
\DeclareMathOperator{\fracc}{FRAC}
\DeclareMathOperator{\stab}{STAB}

\newcommand{\one}{\mathbf{1}} 

\newcommand{\RR}{\mathbb{R}}  

\newcommand{\KK}{\mathcal{K}} 

\title{Lift-and-project ranks of the stable set polytope of joined $a$-perfect graphs}

\author[1]{Bianchi S.}
\author[2]{Escalante M.}
\author[1]{Montelar M.S.}

\affil[1]{Universidad Nacional de Rosario, Argentina}
\affil[2]{CONICET and Universidad Nacional de Rosario, Argentina }
\date{}

\begin{document}
\maketitle



\begin{abstract}
In this paper we study lift-and-project polyhedral operators defined by Lov\'asz and Schrijver  and Balas, Ceria and Cornu\'ejols  on the clique relaxation of the stable set polytope of web graphs. We compute the disjunctive rank of all webs  and consequently of antiweb graphs. We also obtain the disjunctive rank of the antiweb constraints for which the complexity of the separation problem is still unknown. Finally, we use our results to provide bounds of the disjunctive rank of larger classes of graphs as joined $a$-perfect graphs, where near-bipartite graphs belong. 
\end{abstract}
%
%
\section{Introduction}
In this work we study the behavior of lift-and-project operators over the clique relaxation of the stable set polytope of web graphs and use the results for finding lower bounds for the disjunctive rank of a larger classes of graphs as near-bipartite and quasi-line graphs.

Web graphs have circular symmetry of their maximum cliques and stable sets. They belong to the classes of quasi-line graphs and claw-free graphs and are, besides line graphs, relevant for describing the stable set polytope of larger graph classes \cite{GaSa, GiTro,oriolo}.

The lift-and-project operators we analyze are the disjunctive operator defined by Balas, Ceria and Cornu\'ejols in \cite{BCC} and the $N$-operator defined by Lov\'asz and Schrijver in \cite{LoSch91}. It is known that after applying successively these operators to a convex set in $[0,1]^n$ they arrive to the convex hull of integer solutions in this set in at most $n$ iterations \cite{BCC,LoSch91}. These results allow the definition of the \emph{lift-and-project rank} as the minimum number of iterations they need in order to get this convex hull. 

These ranks can be seen as a measure of \emph{how far} a polyhedron is from being integral. One of the main goals of this contribution is to find this measure in the family of web graphs and their complements. In addition, it is interesting to compute bounds for the minimum number of iterations needed for a lift-and-project operator in order to obtain a certain type of inequality to be valid for the corresponding  relaxation. This number, called the \emph{lift-and-project rank of the inequality} can be used for bounding the lift-and-project rank of other families of polyhedra.

In this sense, we compute the disjunctive rank of all webs (see Lemma \ref{rdW2(k+1)+s} and Theorem \ref{rdWnk}). These results also give the disjunctive ranks of their complements, the antiweb graphs. Later we focus on determining the $N$-rank of the inequalities describing the stable set polytope of a particular family of webs (see Theorem \ref{rangodisjP}). Finally, in Theorem \ref{rdFAR} we compute the disjunctive rank of antiweb inequalities and using it, we find bounds for the disjunctive rank of more general families of graphs such as $a$-perfect and joined $a$-perfect graphs.

A preliminary version of some of the results presented in this paper appeared without proofs in \cite{BEM09,BEM_ENDM,BEM_13}.

\section{Preliminaries}

Given $G=(V,E)$ and $v\in V$, the \emph{neighbourhood} of $v$ is $\Gamma(v)=\{u\in V: uv \in E\}$. 
The graph obtained by \emph{deletion} of a set of nodes $U\subseteq V$ is denoted by $G-U$ and corresponds to the subgraph induced by $V \setminus U$. When $U=\{u\}$ we simply write $G-u$. 

A \emph{stable set} in $G$ is a subset of nodes mutually nonadjacent in $G$ and $\alpha(G)$ denotes the cardinality of a stable set of maximum cardinality. A \emph{clique} is a subset of nodes inducing a complete a graph in $G$. We denote by $\omega(G)$ the \emph{clique number} of $G$, the size of a maximum clique in the graph. 

A graph is a \emph{hole} if it is a chordless cycle and is an \emph{antihole}, the complement of a hole. If the number of nodes is odd is called an odd hole or odd antihole, respectively.

The \emph {stable set polytope} of $G$, $\stab(G)$, is the convex hull of the incidence vectors of all stable sets in $G$. A canonical relaxation of $\stab(G)$
 called the \emph{fractional stable set polytope} is defined as
\[
\fracc(G)=\{x\in \RR^V_+: x_i+x_j \leq 1, \text{ for every } ij \in E \}.
\]

A stronger relaxation of the stable set polytope is the \emph{clique relaxation} given by 
\[
\qstab(G)=\{x\in \RR^V_+: x(Q) \leq 1, \text{ for every } Q \text{ clique in } G\}
\]
where $x(U)=\sum\limits_{i\in U} x_i$ for any  $U\subseteq V$. For a  clique $Q$, the inequality $x(Q) \leq 1$ is called a \emph{clique constraint}.

Clearly, $\stab(G)\subseteq \qstab(G)$ for every graph $G$ but equality holds for perfect graphs only \cite {Chv}. A graph is \emph{perfect} if all its node induced subgraphs have the same chromatic and independence numbers \cite{Be60}. 
A graph is called \emph{minimally imperfect}  if it is not perfect but all its proper node induced subgraphs are perfect.
It is known that the only minimally imperfect graphs are the odd holes and their complements \cite{Chud}. 

Web graphs are a natural generalization of the minimally imperfect graphs. More precisely, if $n$ and $k$ are integer numbers with $n\geq 2(k+1)$, a web $W_{n}^k$ is a graph with node set  $\{1,\dots,n\}$ and where $ij$ is an edge if $\left|i-j\right|\leq k$ considering $\{1,\dots,n\}$ as the algebraic group with addition modulo $n$.

If $G=(V,E)$ is a minimally imperfect graph then 
\[
\stab(G)=\qstab(G)\cap \left\{x: x(V)\leq \alpha(G)\right\}
\]
where $x(V) \leq \alpha(G)$ is the \emph{rank constraint} of $G$.

For all imperfect graphs $G$,  $\stab(G)\neq \qstab(G)$ and it is natural to consider the difference between these two polytopes in order to determine \emph{how far} an imperfect graph is from being perfect.
In this context, lift-and-project operators have been widely used in polyhedral combinatorics (see, for instance, \cite{ANE2,BCC,BEN,LiTu03}). 

\textbf{Lift-and-project operators}.  

Starting from a polyhedron $\KK \subseteq [0,1]^n$, these methods attempt to give a description of the convex hull of integer points in it, $\KK^*=\conv(\KK \cap \{0,1\}^n)$ through a finite number of lift-and-project steps.

The \emph{disjunctive} operator is a lift-and-project method which can be characterized  as follows \cite{BCC}: if $j\in \{1,\dots,n\}$, 
\begin{equation}\label{pj}
P_{j}(\KK)=\conv( \KK \cap \{x\in \RR_{+}^{n}:x_{j}\in \{0,1\} \}). 
\end{equation}

The authors prove that this operator can be applied iteratively over a set $F\subseteq \{1,\dots,n\}$ and using \eqref{pj}, it achieves $\KK^*$ in at most $n$ iterations. Then, the \emph{disjunctive rank} of $\KK$,
$r_d(\KK)$, is defined as the smallest cardinality of $F$ for which $P_{F}(\KK) =\KK^*$.

Lov\'asz and Schrijver had previously defined another lift-and-project operator in \cite{LoSch91}, called the $N$-operator.
If $\KK\subseteq [0,1]^{n}$, $\cone(\KK)$ is the polyhedral cone obtained from $\KK$ via homogenization on the new variable $x_0$ (see \cite{Tuncel_bib} for further details).
Let 
\[
\begin{array}{lll}
M(\cone(\KK))= &\{ Y\in\mathbb R^{(n+1)\times(n+1)}:& Y {\rm symmetric }, \; Ye_0=diag(Y), \\
&&  Ye_i\in \cone(\KK), \; Y(e_0-e_i)\in \cone(\KK), \\ 
&& \text{for } \; i=1,\dots,n\},
\end{array}
\]
where $e_{i}$ is  the $i$-th unit vector in $\RR^{n+1}$. 

Projecting this set back onto $\RR^{n+1}$, it results in
\[
N(\cone(\KK))=\{Ye_0:Y\in  M(\cone(\KK))\}.
\]

For simplicity, when we say that we are applying the $N$-operator to some convex set $\KK\subseteq [0,1]^n$, we mean that we consider the cone corresponding to this convex set, apply the lift-and-project procedure, then take the convex subset of $[0,1]^n$ defined by the intersection of this new cone with $x_0=1$. $N(\KK)$ denotes this final subset of $[0,1]^n$ and similarly the relaxations of $\KK^*$ obtained after \emph{applying} this operators in succession.

If $N^r(\KK)$ is the $r$-th iteration of $N$ over $\KK$,  in \cite{LoSch91} it is proved  that
$
N^{n}(\KK) =\KK^*.
$
As for the disjunctive operator, this property allows the definition of $r(\KK)$, the $N$-\emph{rank} of $\KK$, as
the smallest integer $r$
for which $N^{r}(\KK) =\KK^*$. 

It is not hard to see that, for every $j=1,\dots,n$, these relaxations satisfy
\[
\KK^* \subseteq N(\KK) \subseteq \bigcap_{j=1}^n P_{j}(\KK) \subseteq P_{j}(\KK)\subseteq \KK
\]
and  then
\begin{equation}
\label{relacion}
r(\KK)  \leq r_d(\KK).
   \end{equation}

In addition, if  $L$ stands for any of the lift-and-project operators considered herein, the \emph{$L$-rank of a facet constraint} $ax\leq b$ of $\KK^*$ is the minimum number of steps $r$ needed to obtain $ax\leq b$ as a valid inequality for $L^r(\KK)$.

Clearly, if $ax\leq b$ is a facet constraint of $\KK^*$, its $L$-rank is at most the $L$-rank of the relaxation.

In the following sections we apply the above defined lift-and-project operators over the clique relaxation of the stable set polytope in a graph $G$, i.e., $\KK=\qstab(G)$ and $\KK^*=\stab(G)$. Then, in order to simplify the notation, we write $P_j(G)$ and $N^k(G)$ for $P_j(\qstab(G))$ and  $N^k(\qstab(G))$, respectively. Similarly, $r_d(G)$ and  $r(G)$  denote their corresponding ranks and we refer to them as the \emph{disjunctive rank}  and the \emph{$N$-rank} of $G$.

Due to the relationship between the corresponding relaxations we have that, in general,

\begin{equation}\label{rd}
r_{d}(G)\leq r_{d}(\fracc(G))\quad  \text{ and } \quad r(G)\leq r(\fracc(G)).
\end{equation}
Moreover, using the results in \cite{LoSch91}, the $N$-ranks of $\qstab(G)$ and $\fracc(G)$ are related through the clique number of $G$. Actually,   
\begin{equation}\label{frac}
  r(\fracc(G)) \leq  r(G) + \omega(G)-2.
\end{equation}

In \cite{Nasini} it is shown that the disjunctive rank of a graph $G$, namely $r_d(G)$, can be easily described by taking its combinatorial structure into account.

\begin{theorem}[\cite{Nasini}] \label{gra} Given a graph $G$, the disjunctive rank of a graph $G$ coincides with the minimum number of nodes that must be deleted from $G$ in order to obtain a perfect graph.
\end{theorem}

In \cite{ANE2} it is proved that:

\begin{theorem}[\cite{ANE2}]\label{complemento}
The disjunctive rank of a graph coincides with the disjunctive rank of its complement.
\end{theorem}

\section{Lift-and-project operators on the clique relaxation of the stable set polytope of web graphs}\label{capredes}

In this section we present our main results on the behaviour of the disjunctive operator over all webs and compare it with the $N$-operator over a particular family of webs.

Observe that for an integer number $k\geq 2$, $W^1_{2k+1}$ is an odd hole and $W^{k-1}_{2k+1}$ an odd antihole.

If the web $W_{n'}^{k'}$ is a subgraph of $W_n^k$, it is called a \emph{subweb} and we write $W_{n'}^{k'}\subseteq W_n^k$. A subweb is \emph{proper} when it is a proper subgraph, i.e., when $n'<n$. In \cite{Trotter} Trotter presented necessary and sufficient conditions for a graph to be a subweb of a given web. 

The next result provides a characterization of subwebs. 

\begin{lemma}[\cite{Trotter}]\label{Tro1}
Given $k$ and $n\geq 2(k+1)$, the web  $W_{n'}^{k'}$ is a subweb of $W_n^k$ if and only if these numbers satisfy
\[
n\frac{k'}{k}\leq n'\leq \frac{k'+1}{k+1}n.
\]
\end{lemma}

Now we focus on the disjunctive rank of webs. There are particular webs whose disjunctive rank is known. For instance, form Theorem \ref{gra},   $r_d(W_{2p}^1)=0$ since $W_{2p}^1$ is a perfect graph and $r_d(W_{2p+1}^1)=1$ since $W_{2p+1}^1$ is a minimally imperfect graph. 

In what follows we study web graphs $W_n^k$ with $k \geq 2$ and $n \geq 2(k+1)$.  

\begin{lemma}\label{rdW2(k+1)+s}
If  $k \geq 2$ and $0\leq s \leq k$ then $r_d(W_{2(k+1) + s}^k) = s$.  
\end{lemma}
\begin{proof}
Let $G$ be the graph that results after deleting $s$  consecutive nodes of $W_{2(k+1) + s}^k$. Clearly $\bar{G}$, the complement of $G$, is a bipartite graph and then it is perfect. Applying Theorem \ref{gra} and the fact that $G$ results a perfect graph we obtain $r_d(W_{2(k+1) + s}^k)\leq s$.

Now, in \cite{PF} it is proved that if $0 \leq s \leq k$ then $$r(\fracc(W_{2(k+1)+s}^k)) = k+s-1.$$ Using \eqref{rd},  \eqref{frac} and the fact that $\omega (W_n^k)=k+1$, it follows that 
\[r_d(W_{2(k+1)+s}^k)\geq r(\fracc(W_{2(k+1)+s}^k))- (k-1).\]
Therefore $r_d(W_{2(k+1)+s}^k)\geq s$, and the proof is complete.
\end{proof}

The following result provides a general upper bound for the disjunctive rank of webs and, as a consequence, for their $N$-rank.

\begin{lemma}\label{menork}
For every web graph $W_n^k$ with $k \geq 1$ and $n \geq 2(k+1)$, we have that $r_d(W_n^k)\leq k$.
\end{lemma}
\begin{proof}

For every $i\in \{1,\dots,n\}$, let $Q_i=\{i,\dots,i+k\}$ denote the maximum clique starting at node $i$ in the web $W_n^k$, where additions are taken modulo $n$. 

Let us consider the matrix $Q$ whose rows are the incidence vectors of $Q_i$ for all  $i\in\{1,\dots,n\}$. If we delete the columns of $Q$ indexed in the set $\{j,\dots,j+k-1\}$, for any $j\in\{1,\dots,n\}$, the resulting matrix has the consecutive ones property.
Then, the polyhedron 
\[
\{x\in \RR^n_+: x(Q_i)\leq 1 \text{ for } i=1,\dots,n\}\cap \{x: x_i=0 \text{ for } i=j,\dots,j+k-1\} 
\]
is integral for every $j\in\{1,\dots,n\}$. 
Since 
\[
\qstab(W^k_n) \subseteq \{x\in \RR^n_+: x(Q_i)\leq 1 \text{ for } i=1,\dots,n\},
\]
it follows that  
\[
\qstab(W_n^k)\cap \{x: x_i=0 \text{ for } i=j,\dots,j+k-1\}
\]
coincides with 
\[
\stab(W_n^k) \cap \{x: x_i=0 \text{ for } i=j,\dots,j+k-1\}.
\]
If $G_j=W_n^k\setminus\{j,\dots,j+k-1\}$ we have that $\qstab(G_j)=\stab(G_j)$ and then $G_j$ is a perfect graph, for any $j\in\{1,\dots,n\}$. Then, after deleting $k$ consecutive nodes in $W_n^k$ we arrive to a perfect graph.
Applying Theorem \ref{gra} the result follows.
\end{proof}

In addition, the bound in Lemma \ref{menork} is actually achieved by all webs $W_n^k$ with at least $3k+2$ nodes. 

\begin{theorem}\label{rdWnk}
If $k\geq 2$ and $n\geq 3k+2$ then $r_d(W_{n}^k)=k$.
\end{theorem}

\begin{proof}
Assume that $n=sk+r$ with  $r \in\{0,\dots,k-1\}$ if $s \geq 4$ and $r \in \{3,\dots,k-1\}$  if $s = 3$. Let $C_i=\{i,\dots,i+k-1\}$ for every $i\in\{1,\dots,n\}$, where additions are taken modulo $n$.

If we show that, for every  $F\subseteq\{1,\dots,n\}$ with $\left|F\right|=k-1$, $W_n^k \setminus F$ contains a minimally imperfect graph then, due to Theorem \ref{gra}, we obtain $r_d(W_n^k) \geq  k$.  This fact together with Lemma \ref{menork} prove the theorem.

Therefore, the proof of the theorem relies on the following claim.

\begin{claim}
Let $F\subseteq\{1,\dots,n\}$ with $\left|F\right|=k-1$ and $\bar F$ the complement of $F$. Then, there is an odd set $D\subset \bar F$ that induces an odd hole in $W_n^k \setminus F$.
\end{claim}

\textbf{Proof of the claim}

Let us define $D_j= \{j,j+k,j+2k,\ldots, j+ (s-1) k\}$ for each  $j\in \{1,\ldots,n\}$ where addition is modulo $n$.

Observe that $D_j \cup \{j+s k\}=D_j$ for each $j\in \{1,\ldots,n\}$ if and only if $r=0$ (since in this case $j+s k = j (mod\;n)$).

Let us define $L_j=D_j \cup \{j+s k\}$ for $j\in\{1,\dots,r\}$.

It is clear that the set $\{1,\dots, n\}$ can be partitioned into the  following $k$ sets: $L_j$ for $j\in \{1,\dots,r\}$ and $D_j$ for $j\in\{r+1,\dots,k\}$.

According to the Pigeonhole Principle there is $i\in \{1,\ldots, k\}$ such that either $L_i\cap F=\emptyset$ if $i\in \{1,\ldots,r\}$ or $D_i \cap F=\emptyset$ if $i\in \{r+1,\ldots,k\}$.

Then, in what follows we study the only two possibilities; i.e., the case when $L_i\cap F=\emptyset$ for some $i\in \{1,\dots, n\}$ and the case when  $L_i\cap F\neq \emptyset$ for all $i\in \{1,\dots,n\}$.

\begin{itemize}
\item[a)]
Let us first consider the case when   $L_i\cap F=\emptyset$ for some $i\in \{1,\ldots, n\}$. This case includes either $r=0$ and $s\geq 4$ or $r\neq 0$ and $s\geq 3$, thus $\left|D_i \right|\geq 4$.
   
Then if $\left|L_i\right|$ is odd the set $D=L_i\subset \bar F$ induces an odd hole in $W_n^k$ and the claim follows.

Assume that $\left|L_i \right|$ is even.  

If there is $t\in D_i$ such that $\left|C_t\cap F\right|=k-1$ then $\bar F= (\{1,\ldots,n\}\setminus C_t) \cup \{t\}$.
In this case there are many ways to find a subset $D$ of nodes inducing an odd hole in $W_n^k\setminus F$. For instance, if $t\notin\{i+(s-2)k, i+(s-1)k\}$ then $D=(D_i\setminus\{t+2k\})\cup \{t+2k-1,t+2k+1\}$. Otherwise if $t\in\{i+(s-2)k, i+(s-1)k\}$ then $D=(D_i\setminus\{t-k\})\cup \{t-1,t-k-1\}$. Hence, the claim follows.

Now, consider $\left|C_t\cap F\right|<k-1$ for all $t\in D_i$. Let  $i+l\in C_i\cap \bar F$ be such that $\{i+l+1,\ldots,i+k-1\}\subset F$, i.e., $i+l$ is the farthest node from node $i$ that also belongs to $\bar F\cap C_i$. Observe that $l\in \{1,\ldots,k-1\}$.

If $\left|C_{i+l+1}\cap F\right|<k-1$ then there is $m\in\{1,\dots,l\}$ such that $i+k+m\in \bar F$. Hence the set $D=(D_i\setminus \{i+k\})\cup \{i+l,i+k+m\}$ induces an odd hole in $W_n^k\setminus F$. 

Otherwise, if $\left|C_{i+l+1}\cap F\right|=k-1$ then the set $D=(D_i\setminus \{i+2k\})\cup \{i+2k-1,i+2k+1\}$ induces the odd hole needed to prove the claim.

\item[b)]
Assume that $L_i\cap F\neq \emptyset$ for all $i\in \{1,\ldots,n\}$. Then $r\neq 0$.
W.l.o.g assume that $D_1 \cap F=\emptyset$ and $\{1+s k\}\in F$. Then, it holds that $L_i\cap F\neq \emptyset$ for $i\in\{2,\dots,r\}$.
Again by the pigeonhole principle there must be $j\in\{r+1,\dots,k\}$ such that $D_j \cap F=\emptyset$ and in this case $j+sk \in F$. Observe that $j+sk = j-r\, mod \; n$ and $1<j-r<r$. Thus $j+sk \in C_1$ and $|1-(j+(s-1)k)|<k$. 

Let $D'=\{1\}\cup D_j$. If $|D'|$ is odd then we can consider $D=D'$ and the claim follows.

On the other hand, if $|D'|$ is even, since $\{1+s k, j+sk\}\in F$, then $|C_j\cap F|\leq k-3$. Therefore there must be $m\in\{1,\dots,k\}$ such that $j+m \in \bar F\cap C_{1+k}$.
Hence we can define $D=\{1,j+m,1+2k\} \cup (D_j \setminus \{j+k\})$ and the proof is complete.
\end{itemize}
\end{proof}

\begin{remark}\label{punto}
In \cite{BEM_ENDM} we proved that the $N$-rank of the web $W_{s(k+1)+k}^k$  is also $k$. For this purpose, we showed the existence of a point in $ N^{k-1}(W_{s(k+1)+k}^k)$ violating the rank inequality, valid for $\stab(W_{s(k+1)+k}^k)$.

This result, together with the previous theorem exhibits an infinite  family of webs where the two ranks coincide, i.e., $r_d(W_{s(k+1)+k}^{k})=r(W_{s(k+1)+k}^{k})$, for $s,k\geq 2$.
\end{remark}

Next, we see that most of web graphs have a member of the above family as a subweb thus giving a lower bound for the $N$-rank. More precisely,

\begin{corollary}\label{cotaN}
Let $n=s(k+1)+r$ with $k\geq 2$, $s\geq 3$ and $0\leq r\leq k-1$. Then $r(W_n^k)\geq k-t$ where $t=\left\lceil\frac{k(1+r)}{r+s} \right\rceil$. 
\end{corollary}
\begin{proof}
Firstly observe that if $t=\left\lceil\frac{k(1+r)}{r+s} \right\rceil$ then $t\leq k-1$.

After Trotter's formula, it is easy to prove that if $k'=k-t$ and  $n'=(s-1)(k'+1)+k'$ then $W^{k'}_{n'}$ is a subweb of $W_n^k$. According to Remark \ref{punto} we obtain that $r(W_{n'}^{k'})=k'$ and then $k-t\leq r(W_n^k)$.
\end{proof}

Let us now make use of these results in order to determine the rank of the complementary graphs of webs, called \emph{antiwebs}. For simplicity, we denote by $A_n^k$ the antiweb obtained as the complement of $W_n^{k-1}$.
Theorems \ref{complemento} and \ref{rdWnk} allow us to compute the disjunctive rank of antiwebs.

\begin{corollary}\label{antiweb}
Let $k \geq 2$. If $s\in\{0,\dots,k\}$ then $r_d(A_{2(k+1) + s}^{k+1}) = s$ and if $n\geq 3k+2$ then $r_d(A_{n}^{k+1})=k$.
\end{corollary}

\section{Lift-and-project rank of facets of the stable set polytope of webs} 

Let us consider the rank constraint associated with the web $W_{s(k+1)+k}^k$, mentioned in Remark \ref{punto}, a particular family where the disjunctive and the $N$-rank coincide. 


\begin{lemma}\label{fullrank}
If  $\pi$ is the rank constraint associated with  $W_{s(k+1)+k}^k$, i.e.,
\[ \pi: x(V(W_{s(k+1)+k}^k ))\leq  s,\]
then $r(\pi) = r_d(\pi) = k$.
\end{lemma}
\begin{proof}

In \cite{BEM_ENDM} it is proved the existence of a point  $\bar{x}\in N^{k-1}(W_{s(k+1)+k}^k)$ violating the rank constraint. This shows that $r(\pi)\geq k$.

From Theorem \ref{rdWnk} we have that $r_d(\pi) \leq k$. Since $r(\pi) \leq r_d(\pi) \leq k$, the result follows. 
\end{proof}

Dahl in \cite{Dahl} characterizes the facet defining inequalities of $\stab(W_n^2)$ for $n \geq 6$,  by introducing $1$-interval inequalities.

Let us consider $W_n^2$, for $n \geq 6$.
If $V = \{1, \dots, n\}$ an \emph{interval} is a subset of consecutive nodes using modulo $n$ arithmetic. For example, the set $\{n-2, n-1,n, 1\}$ is an interval. 
A set $T \subsetneq V$  is a \emph{$1$-interval set} if there is a partition of $V$ into a collection of disjoint intervals $I_1,J_1,\dots,I_t,J_t$ where $T=\bigcup_{j=1}^t I_j$ and $\left|J_j\right|=1$ for every $j= 1, \dots, t$. 

Given a $1$-interval set $T$, the \emph{$1$-interval inequality} associated with $T$  is $x(T) \leq \alpha(T)$ where  $\alpha(T)$ is the stability number of the subgraph of $W_n^2$ induced by the nodes of $T$ (see \cite{Dahl} for further details). 

\begin{theorem} [\cite{Dahl}]\label{W-n-2}
For every $n \geq 6$, $\stab(W_n^2)$ is described by 
\begin{enumerate}
	\item non-negativity constraints,
	\item clique constraints,
	\item the rank constraint when  $n$ is not a multiple of  $3$,
	\item $1$-interval inequalities associated with $T\subsetneq V$ such that $\left|I_j\right|= 1 \;\rm{ mod } \; 3$ for $j=1, \dots, t$ and $t\geq 3$ odd. 
\end{enumerate}
\end{theorem}

Let us first compute the disjunctive rank of the rank constraint associated with $W_n^2$. 
\begin{lemma}\label{rangodisjP1}
If $\pi$ is the rank constraint associated with $W_{3s+\ell}^2$ for $\ell\in \{0,1,2\}$, i.e.,
\[ \pi: x(V(W_n^2))\leq  s,\]
then $r(\pi) = r_d(\pi) = \ell$.
\end{lemma} 
\begin{proof}
Let us recall that $Q_i$ denotes the clique of the $k+1$ consecutive nodes in the web starting at node $i$.
Then, consider the clique constraints 
$x(Q_{3j+1})\leq 1$ for $j=0, \dots, s-1$. If we sum them up, we obtain
\begin{equation}\label{sumaQ}
\sum_{j=0}^{s-1} x(Q_{3j+1}) = \sum_{i=1}^{3s}x _i \leq s,
\end{equation}
a valid inequality for $\qstab(W_n^2)$.

If $\ell=0$ then $\pi$ is obtained by a linear combination of the clique constraints, showing that both, the disjunctive and the $N$ rank,  are equal to zero.

Now, if $\ell = 1$ the point $x=\frac{1}{3}\one \in \qstab(W_{3s+1}^2)$ violates the rank inequality $\pi$. Therefore any of the ranks is at least one. 

On the other hand, \eqref{sumaQ} is valid for $\qstab(W_{3s+1}^2)\cap \{x:x_{3s+1}=0\}$. Hence $r_d(\pi)=1$ and then $r(\pi)=1$.

Finally, if $\ell=2$ the web is $W^2_{3s+2}$ and Lemma \ref{fullrank} shows that $r(\pi)=r_d(\pi)=2$.
\end{proof}

In addition, 
\begin{theorem}\label{rangodisjP}
Every 1-interval facet defining inequality for the stable set polytope of $W_n^2$ has disjunctive and $N$-rank equal to one.
\end{theorem} 

\begin{proof} 
Let $T = \cup_{j=1}^t  I_j$ and $|I_j|=3k_j + 1$ for some integer $k_j$, $j\in\{1,\dots,t\}$ for $t$ odd. Let $\pi_T:  x(T) \leq \alpha(T)$ be the corresponding 1-interval facet defining inequality for $\stab(W^2_n)$.  Clearly, its rank is at least 1 since this facet is absent in the original relaxation of $\stab(W^2_n)$. 

If $x \in \qstab(W_n^2)$ then it satisfies $x(I_j) \leq k_j +1$ for each $j=1,\dots,t$. In addition, 
\[
x(I_{j-1}\cup I_{j}) \leq  k_{j-1} + 1 + k_{j}
\]
for every $j\in\{2,\dots,t\}$.

Therefore, since $t$ is odd, it follows that $x \in \qstab(W_n^2)$ satisfies 
\begin{equation}\label{suma}
\sum_{i\in T} x_i\leq \sum_{j=1}^{t-1} k_j +  \frac{t-1}{2} + k_t + x_{3k_t+1}.
\end{equation} 

Moreover,  
\begin{equation}\label{alfa}
\alpha(T)=\sum_{j=1}^{t} k_j +  \frac{t-1}{2}.
\end{equation} 

According to \eqref{suma} and \eqref{alfa}, if $\bar{x}\in\qstab(W_n^2)\cap \{x:x_{3k_t+1}=0\}$
then  $\bar{x}(T)\leq \alpha(T)$.

This shows that the 1-interval inequality $\pi_T$ is valid for $P_{3k_t+1}(W^2_n)$ and then $r_d(\pi_T) \leq 1$. This completes the proof since $1\leq r(\pi_T)\leq r_d(\pi_T)$.
\end{proof}

Observe that, according to Theorem \ref{W-n-2},  all the inequalities describing $\stab(W_n^2)$ are obtained in at most one step of the $N$-operator when $n$ is either $3s$ or $3s+1$ for some $s\geq 2$. 
Hence we have the following consequence.
\begin{corollary}\label{rangoNP3s}
For every $s\geq 2$, $r(W^2_{3s})=r(W^2_{3s+1})=1$. 
\end{corollary}
\begin{proof}
Using Lemma \ref{rangodisjP1} and Theorem \ref{rangodisjP} we have that the rank constraint and the $1$-interval inequalities are valid for $N(W_n^2)$ when $n\in \{3s,3s+1\}$. According to Theorem \ref{W-n-2}, these inequalities together with the inequalities in $\qstab(W_n^2)$ are enough to describe $\stab(W_n^2)$. From the fact that $\stab(W_n^2)\subseteq N(W_n^2)\subseteq \qstab(W_n^2)$ and convexity arguments, the corollary follows.
\end{proof}

Although Lemma \ref{rangodisjP1} and Theorem \ref{rangodisjP} prove that the disjunctive rank of all inequalities describing $\stab(W_n^2)$ is also one when $n\in\{3s,3s+1\}$, in Theorem \ref{rdWnk} we obtained that $r_d(W_n^2)=2$ if $n\geq 8$.

Nevertheless, from Remark \ref{punto} when $k=2$ and the previous corollary, we have:
\begin{corollary}
If $n\geq 8$, the disjunctive and the $N$-ranks  of $W_n^2$ coincide if and only if $n=3s+2$ for some $s\geq 2$.
\end{corollary}

\section{Joined $a$-perfect graphs}

Let us recall that $A_n^k$ denotes the complement of the web graph $W_n^{k-1}$. The rank constraint of the antiweb  $A_n^k$ 
\begin{equation}\label{antiwebcons}
 x(V(A_n^k)) \leq k
\end{equation}
is called an \emph{antiweb constraint}.

In \cite{Trotter} Trotter shows that the constraint (\ref{antiwebcons}) 
defines a facet of $\stab(A_n^k)$ if and only if $n$ and $k$ are relatively prime numbers. In this case, the antiweb is called \emph{prime} and the inequality \eqref{antiwebcons} is a \emph{prime antiweb constraint}.

Later, Wagler in \cite{AWaw} proves that  $\stab(A_n^k)$ is completely described by non-negativity constraints and prime antiweb constraints associated with subantiwebs in $A_n^k$.
In addition, the author defines several graph classes where the inequality in \eqref{antiwebcons} plays an important role.

A graph $G$ is \emph{$a$-perfect} if $\stab(G)$ is described by non-negativity and prime antiweb constraints only (see \cite{AWaw}). 
Observe that perfect graphs and antiwebs are $a$-perfect graphs. 

In this section we exhibit bounds for the disjunctive rank of  $a$-perfect graphs by means of the disjunctive rank of antiweb constraints.
 
Let us first show the following result.
\begin{theorem}\label{rdFAR}
Let $k\geq 1$ and $n \geq 2k$. If $A_n^k$ is a prime antiweb and $\pi: x(V(A_n^k))\leq k$ stands for its rank constraint, then 
\[
r_d(\pi)=n-\omega k,
\]
where $\omega=\left\lfloor \dfrac{n}{k}\right\rfloor$  is the clique number of $A_n^k$. 
\end{theorem}
\begin{proof}
Let us denote $\beta=n-\omega k$. Clearly, $\beta \in \{1,\dots, k-1\}$ since $n$ and $k$ are relatively prime numbers. Let $F= \{ \omega k + 1,\dots,\omega k + \beta\}$ and consider the $k$ maximal cliques $Q_i=\{i, i + k, \dots, i + (\omega-1)k \}$ for $i\in \{1,\dots,k\}$. Then, the set $\{F, Q_1, \dots, Q_k\}$ defines a partition of $V(A_n^k)=\{1,\dots,n\}$. Also, if $x\in \qstab(A_n^k)$ then 
\[ x(V(A_n^k)) = x(F)+ \sum_{i=1}^{k} x(Q_i) \leq x(F) + k.\] 
It holds that  $\pi$ is a valid inequality for 
\[\qstab(A_n^k)\cap \{x: x_i=0 \; \text{for } i\in F\}.\]
Hence,  $r_d(\pi) \leq \beta$.   

Now, let $T \subseteq V(A_n^k)$ with $|T|=\beta - 1$ and let $\bar{x} \in \RR^{n}$ be such that
\[\bar{x}_i=\left\{
\begin{array}{ll}
0 & \text{if } i\in T,\\
&\\
\dfrac{1}{\omega} & \text{otherwise}.
\end{array}
\right.
\]

Clearly, 
\[
\bar{x}(V(A_n^k))=  (\omega k + 1)\frac{1}{\omega} = k + \frac{1}{\omega}>k.
\] 

This shows that $\bar{x} \in P_T( A_n^k)$ violates $\pi$ and then, $r_d(\pi) \geq \beta$.
This completes the proof.
\end{proof}

The theorem above allows us to present a lower bound for the disjunctive rank of $a$-perfect graphs.
\begin{corollary}\label{a-perfect}
Let $G$ be an $a$-perfect graph. If  $A_{n_i}^{k_i}$ is a prime antiweb in $G$ and $\omega_i=\omega(A_{n_i}^{k_i})$ for $i \in I$ then
\begin{equation}\label{cota-a}
r_d(G)\geq \max\left\{n_i-\omega_ik_i:  \ i \in I\right\}.
\end{equation}
\end{corollary}

\begin{remark}
Note that in the bound given in Corollary \ref{a-perfect} we have to consider all the node induced prime antiwebs in the given graph $G$. In fact, if $A_{n'}^{k'}$ is a subgraph of $A_n^k$ and $\omega(A_{n'}^{k'}) =\omega(A_n^k)$ then, using Lemma \ref{Tro1}, it holds that $ \frac{n'}{k'} \leq \frac{n}{k} $. Therefore, 
\[
n'-\omega k' = k' (\frac{n'}{k'} - \omega) \leq k' (\frac{n}{k} - \omega) < k (\frac{n}{k} - \omega)= n - \omega k.
\]
Thus, $r_d(A_{n'}^{k'}) < r_d(x(V(A_{n}^{k}))\leq k)$.

However, if the clique numbers do not coincide the same result may not hold. For example, $A_{17}^3$ is a subantiweb of $A_{25}^4$, where $r_d(x(V(A_{25}^4))\leq 4)=1$ and $r_d(x(V(A_{17}^3))\leq 4)=2$. 
\end{remark}

In \cite{Wa_habil}, Wagler defines another graph class where its members are obtained by using the complete join operation between antiwebs.

Given two graphs $G_1=(V_1,E_1)$ and $G_2=(V_2,E_2)$, the complete join of $G_1$ and $G_2$, denoted by $G_1 \vee G_2$, is the graph having node set $V_1\cup V_2$ and edge set $E_1\cup E_2 \cup \{uv: u\in V_1, v\in V_2\}$. 

Chv\'atal, in \cite{Chv}, obtained facets of the stable set polytope of a complete join of graphs after the facets of the stable set polytopes of the corresponding graphs. More precisely, 

\begin{lemma}[\cite{Chv}] If $\pi_i: a_i\: x(V(G_i))\leq 1$ defines a facet of $\stab(G_i)$ for $i=1,2$ then
\begin{equation}\label{pi}
\pi: a_1\: x(V(G_1)) + a_2\: x(V(G_2))\leq 1
\end{equation}
defines a facet of $\stab(G_1\vee G_2)$.
\end{lemma}

The inequality in \eqref{pi} is called \emph{joined inequality} associated with $\pi_1$ and $\pi_2$.

A graph $G$ is \emph{joined $a$-perfect} if  $\stab(G)$ is described by non-negativity constraints and joined inequalities associated with prime antiwebs in $G$, i.e.,
\begin{equation}\label{joinanti-red}
\sum_{i=1}^t\frac{1}{\alpha(A_i)} x(V(A_i))\leq 1
\end{equation}
where $A_1, \dots, A_t$ are prime antiwebs in  $G$ such that  $A_1\vee \dots \vee A_t \subseteq G$.

Hence we have:
\begin{theorem}\label{lema_join}
Let  $G=A_1 \vee A_2$ where $A_1$ and $A_2$ are prime antiwebs. Let $\pi_i$ be the rank constraint associated with $A_i$ for $i=1,2$ and $\pi$ their joined inequality, i.e.,
\begin{equation}
\pi: \frac{1}{\alpha(A_1)}\: x(V(A_1)) + \frac{1}{\alpha(A_2)}\: x(V(A_2))\leq 1.
\end{equation}
Then, $r_d(\pi)\geq r_d(\pi_1)+ r_d(\pi_2)$.
\end{theorem}
\begin{proof}
Let $F \subseteq V(G)$. If $|F| < r_d(\pi_1) + r_d(\pi_2)$ then $|F \cap V(A_i)| < r_d(\pi_i)$ for some   $i\in \{1,2\}$. W.l.o.g. assume that $|F \cap V(A_1)| < r_d(\pi_1)$. 

Then, there exists
\[\bar{x} \in  \qstab (A_1 ) \cap \{x \in \RR^{V(A_1)}:x_i=0 \text{ for } i \in F \cap V(A_1) \}\] 
such that $\bar{x}$ violates $\pi_1$, or equivalently, $\bar{x}(V(A_1)) > \alpha(A_1)$.  

Consider $\tilde{x} \in \RR^{|V|}$ defined by 
\[\tilde{x}_i = \left\{
\begin{array}{lcl}
\bar{x}_i & \text{if} & i \in V(A_1) \\
0 & \text{if} & i \in V(A_2).
\end{array} \right.
\]
Clearly $\tilde{x} \in \qstab(G)$ and  
\[\frac{1}{\alpha(A_1)}\: \tilde{x}(V(A_1)) + \frac{1}{\alpha(A_2)}\: \tilde{x}(V(A_2)) =  \frac{1}{\alpha(A_1)}\: \bar{x}(V(G)) >  \frac{1}{\alpha(A_1)} \alpha(A_1),
\] 
that is, $\tilde{x}\in P_F(G)$ and it violates the inequality $\pi$. 

Therefore, $r_d(\pi)\geq r_d(\pi_1)+ r_d(\pi_2)$.
\end{proof}

This result gives a bound for the disjunctive rank of joined $a$-perfect graphs.

\begin{corollary}\label{cotajoin}
Let $G$ be a joined $a$-perfect graph and $A_{n_i}^{k_i} \subseteq G$ a prime antiweb, for every $i\in I$. If $\omega_i=\omega(A_{n_i}^{k_i})$ for $i \in I$ then 
\begin{equation}\label{cota}
r_d(G)\geq \max\left\{\sum_{i\in S}(n_i-\omega_ik_i): \bigvee_{i\in S} A_{n_i}^{k_i} \subseteq G, \text{ for } S\subseteq I\right\}.
\end{equation}
\end{corollary}
\begin{proof}
If $\pi$ is a nontrivial facet of $\stab(G)$ then, since $G$ is a joined $a$-perfect graph,  
\[
\pi: \sum_{i\in S}\frac{1}{\alpha(A_{n_i}^{k_i})} x(V(A_{n_i}^{k_i}))\leq 1
\]
for some $S \subseteq I$ such that $\bigvee_{i\in S} A_{n_i}^{k_i} \subseteq G$.

From Theorem \ref{rdFAR},  $r_d(A_{n_i}^{k_i})=n_i-\omega_ik_i$, for all $i \in I$. Applying Theorem \ref{lema_join} it holds that 
\[r_d(G)\geq  r_d(\pi) \geq \sum_{i\in S} (n_i-\omega_ik_i)\] and then the result follows.
\end{proof}

This last result helps us to compute bounds for the disjunctive rank of larger classes of graphs, such as near-bipartite graphs and their complements, the quasi-line graphs.
A graph $G$ is \emph{near-bipartite} if the graph obtained after deleting any node and all its neighbors, is a bipartite graph.  
If $G$ is  near-bipartite, its complement has the property that the neighborhood of any of its nodes can be partitioned into two cliques. The graphs with this property are called \emph{quasi-line} graphs.

Using the results due to Shepherd in \cite{She95} we have that near-bipartite graphs are joined $a$-perfect graphs.
\begin{theorem} [\cite{She95}] The only nontrivial facets describing the stable set polytope of a near-bipartite graph are inequalities associated with  join of cliques and prime antiwebs in $G$.
\end{theorem}

As a consequence of Theorem \ref{complemento} we can extend the result obtained in Corollary \ref{cotajoin} to  quasi-line graphs. 

\section{Conclusions} 

In this paper we have exactly determined the disjunctive rank of all webs, and thus, according to Theorem \ref{gra}, of their complements, the antiwebs.
Although, in general, the $N$-operator is much stronger than the disjunctive operator, we give evidence that they do not differ \emph{too much} in the family of web graphs.
In fact, we have presented an infinite family of web graphs where they coincide, and, when $n$ is large enough, they can differ in at most one unit (see Corollary \ref{cotaN}.

The importance of this result relies on the fact that computing the disjunctive rank of a graph is easier than the $N$-rank and after applying the disjunctive procedure the convex obtained preserves the combinatorial properties of the problem.

In addition we have exactly determined the disjunctive rank of antiweb inequalities for which the complexity of the separation problem is still unknown.
The importance of the result in Theorem \ref{rdFAR} is that we have identified the set of indices where we can apply the disjunctive operator for finding an antiweb inequality as a valid inequality for the stable set polytope.

Finally, the result in Corollary \ref{cotajoin} helps us to compute bounds for the disjunctive rank of larger classes of graphs, such as near-bipartite graphs and their complements, the quasi-line graphs.

\end{document}